\documentclass[11pt,reqno,a4paper]{amsart}
\usepackage[british]{babel}

\usepackage{multicol}
\usepackage{graphicx}
\usepackage{amscd}
\usepackage{amsmath}
\usepackage{amsfonts}
\usepackage{amssymb}
\usepackage{color}
\usepackage{pdfsync}
\usepackage{bbm}
\usepackage{dsfont}
\usepackage{graphicx}
\usepackage{hyperref}
\usepackage{comment}

\definecolor{trp}{rgb}{1,1,1}

\definecolor{red}{rgb}{1,0,.2}

\newtheorem{theorem}{Theorem}[section]
\theoremstyle{plain}
\newtheorem{acknowledgement}{Acknowledgment}

\newtheorem*{condition*}{Condition}

\newtheorem{cor}[theorem]{Corollary}

\newtheorem{definition}[theorem]{Definition}

\newtheorem{lemma}[theorem]{Lemma}

\newtheorem{prop}[theorem]{Proposition}
\newtheorem{remark}{Remark}

\numberwithin{equation}{section}

\newcommand{\R}{\mathbb{R}}

\newcommand{\ii}{\mathbf{i}}
\newcommand{\il}{\overline{\imath}}
\newcommand{\jl}{\overline{\jmath}}
\newcommand{\jj}{\mathbf{j}}

\newcommand{\spt}{\mathrm{spt}}

\newcommand{\essinf}{\mathrm{essinf}}
\newcommand{\x}{\underline{x}}
\newcommand{\y}{\underline{y}}
\newcommand{\n}{\underline{n}}
\newcommand{\z}{\underline{z}}
\newcommand{\vv}{\underline{v}}

\begin{document}
\title[Non-linear projections of self-similar sets]{On some non-linear projections of self-similar sets in $\R^3$}

\author{Bal\'azs B\'ar\'any}
\address[Bal\'azs B\'ar\'any]{Budapest University of Technology and Economics, BME-MTA Stochastics Research Group, P.O.Box 91, 1521 Budapest, Hungary \&
Mathematics Institute, University of Warwick, Coventry CV4 7AL, UK}
\email{balubsheep@gmail.com}

\thanks{The research of B\'ar\'any was supported by the grants EP/J013560/1 and OTKA K104745.}

\begin{abstract}
In the last years considerable attention has been paid for the orthogonal projections and non-linear images of self-similar sets. In this paper we consider homothetic self-similar sets in $\R^3$, i.e. the generating IFS has the form $\left\{\lambda_i\underline{x}+\underline{t}_i\right\}_{i=1}^q$. We show that if the dimension of the set is strictly bigger than $1$ then the image of the set under some non-linear functions onto the real line has dimension $1$. As an application, we show that the distance set of such self-similar sets has dimension $1$. Moreover, the third algebraic product of a self-similar set with itself on the real line has dimension $1$ if its dimension is at least $1/3$.
\end{abstract}


\maketitle

\thispagestyle{empty}

\section{Introduction and Statements}

\ \ We call a non-empty compact set $\Lambda$ self-similar in $\R^d$ if there exists an iterated function system (IFS) $\Phi$ of the form

\begin{equation}\label{edefIFS}
	\Phi=\left\{f_i(\underline{x})=\lambda_iO_i\underline{x}+\underline{t}_i\right\}_{i=1}^q,
\end{equation}
where $\lambda_i\in(0,1)$, $\underline{t}_i\in\R^d$ and $O_i$ is an orthogonal transformation of $\R^d$ for every $i=1,\dots,q$, and $\Lambda$ is the attractor of $\Phi$, i.e. the unique non-empty compact set $\Lambda=\bigcup_{i=1}^qf_i(\Lambda)$. We call a measure $\mu$ self-similar if there exists an IFS $\Phi$ in the form \eqref{edefIFS} and a probability vector $(p_1,\dots,p_q)$ such that $\mu=\sum_{i=1}^qp_i(f_i)_*\mu$, where $(f)_*\mu=\mu\circ f^{-1}$.

Let us denote the set of orthogonal projections from $\R^d$ to $\R^k$ by $\Pi_{d,k}$. The classical results of Marstrand~\cite{Mar} and Kaufman~\cite{K} states that for any $A\subseteq\R^d$ Borel set $\dim_H\pi A=\min\left\{k,\dim_HA\right\}$ for almost every $\pi\in\Pi_{d,k}$, where $\dim_H$ denotes the Hausdorff dimension. Let us denote the packing dimension by $\dim_P$ and the box dimension by $\dim_B$. For the definition and basic properties of Hausdorff, packing and box dimension we refer to \cite{F}.

Hochman and Shmerkin \cite{HS} proved that if the IFS $\Phi$ satisfies the strong separation condition (SSC), i.e. $f_i(\Lambda)\cap f_j(\Lambda)=\emptyset$ for every $i\neq j$ and the orthogonal transformations of the IFS $\Phi$ satisfies a minimality assumption, that is there exists a $\pi\in\Pi_{d,k}$ such that
\begin{equation}\label{eminimality}
\bigcup_{n=1}^{\infty}\left\{\pi O_{i_1}\cdots O_{i_n}:1\leq i_1,\dots,i_n\leq q\right\}
\end{equation}
is dense in $\Pi_{d,k}$, then $\dim_H\pi\Lambda=\min\left\{k,\dim_H\Lambda\right\}$ for every $\pi\in\Pi_{d,k}$, moreover, $\dim_Hg(\Lambda)=\min\left\{k,\dim_H\Lambda\right\}$ for every $g\in C^1(\R^d\mapsto\R^k)$ without singular points. In particular, if the minimality assumption holds then \eqref{eminimality} holds for all $\pi\in\Pi_{d,k}$. Recently, Farkas \cite{Fa} generalized this result by omitting the strong separation condition.

Dekking~\cite{D}, Rams and Simon~\cite{RS1,RS2}, Falconer and Jin~\cite{FJ} considered the orthogonal projections and non-linear images of random self-similar sets. For more detailed surveys on projections of fractal sets and measures, see \cite{FFJ} or \cite{Sh}.

In this paper, we focus on \textit{homothetic self-similar sets} (HSS set) in $\R^3$, which is $O_i=I$ for every $i=1,\cdots,q$, where $I$ denotes the identity. Similarly, we consider homothetic self-similar measures (HSS measure).

It is well known fact that in this case the dimension may drop under some orthogonal projections. However, if $\Lambda$ is a HSS set with SSC on $\R^2$ then $\dim_Hg(\Lambda)=\min\left\{1,\dim_H\Lambda\right\}$ for certain $g:\R^2\mapsto\R$ $C^2$ functions. This result was first published in the paper of Bond, \L aba and Zahl~\cite[Proposition~2.6]{BLZ}, but they attribute the proof to Hochman.

Our goal is to generalize this result for HSS sets in $\R^3$, at least in the case when $\dim_H\Lambda$ is large enough.

During the paper we will have a special interest on the radial projection $P_d:\R^d\setminus\{\underline{0}\}\mapsto S^{d-1}$, where $S^{d-1}$ denotes the unit sphere in $\R^d$. Precisely, $P_d(\x)=\frac{\x}{\|\x\|}$. For simplicity, denote the gradient vector of a function $g:\R^d\mapsto\R$ at a point $\x$ by $\nabla_{\underline{x}}g$. 

\begin{theorem}\label{tmain}
	Let $\Lambda$ be an homothetic self-similar set in $\R^3$ such that $\dim_H\Lambda>1$ and $\Lambda$ is not contained in any plane (but not necessarily satisfying SSC). Suppose that $g:\R^3\mapsto\R$ is a $C^1$ function on a $V\supseteq\Lambda$ open set such that
	\begin{enumerate}
		\item $\|\nabla_{\underline{x}}g\|\neq0$ for every $\underline{x}\in\Lambda$,\label{tcond1}
		\item $\|\nabla_{t\cdot\underline{x}}g\times\nabla_{\underline{x}}g\|=0$ for every $\underline{x}\in V$ and for every $t\in\R$ such that $t\cdot\underline{x}\in V$.\label{tcond2}
		\item The function $h_g$ is bi-Lipschitz on $P_3(\Lambda)\subseteq S^2$, where $h_g(\x)=P_3(\nabla_{t\cdot\x}g)$ for any $t\in\R$ such that $t\cdot\x\in V$.\label{tcond3}
		
	\end{enumerate}
Then $\dim_Hg(\Lambda)=1$.
\end{theorem}

We apply Theorem~\ref{tmain} in two ways. First, we show a corollary for the distance set of HSS sets in $\R^3$. Let us denote the {\it distance set} of $A\subset\R^d$ by $D(A)$. That is,
\begin{equation}\label{edistset}
D(A)=\left\{\|\underline{x}-\underline{y}\|:\underline{x},\underline{y}\in A\right\}.
\end{equation}
For every $\underline{x}\in A$, we define the {\it pinned distance set} of $A\subset\R^d$ at the point $\underline{x}$ by
\begin{equation}\label{edistset2}
D_{\underline{x}}(A)=\left\{\|\underline{x}-\underline{y}\|:\underline{y}\in A\right\}.
\end{equation}
Falconer's distance set conjecture states that if $\dim_HA>d/2$ then $D(A)$ has positive Lebesgue measure for any measurable $A\subseteq\R^d$. Recently, Orponen \cite{O} showed that for any self-similar set $\Lambda$ in $\R^2$ if $\mathcal{H}^1(\Lambda)>0$ then  $\dim_HD(\Lambda)=1$, where $\mathcal{H}^1$ denotes the Hausdorff measure. We improve Orponen's result for HSS sets in $\R^3$ in the following way.

\begin{theorem}\label{tdistance}
Let $\Lambda$ be an HSS set in $\R^3$ such that $\dim_H\Lambda>1$. Then for every $\underline{x}\in\Lambda$, $\dim_HD_{\underline{x}}(\Lambda)=1$. In particular, $\dim_HD(\Lambda)=1$.
\end{theorem}

As a second application, we consider the algebraic product of a self-similar set on the real line with itself. Let $A,B\subset\R$ and denote $A\cdot B$ the algebraic product $A$ and $B$, that is,
$$A\cdot B=\left\{x\cdot y:x\in A\text{ and }y\in B\right\}.$$

As a consequence of the result of Bond, \L aba and Zahl \cite{BLZ} we show that for every $\Lambda$ self-similar set on the real line
\begin{equation}\label{eprod2}
\dim_H\Lambda\cdot\Lambda=\min\left\{2\dim_H\Lambda,1\right\},
\end{equation}
see Corollary~\ref{cprod2}. We generalize this result for $\Lambda\cdot\Lambda\cdot\Lambda$ in the following way.

\begin{theorem}\label{tprod}
	Let $\Lambda$ be a self-similar set in $\R$ such that $\dim_H\Lambda>1/3$. Then $\dim_H\Lambda\cdot\Lambda\cdot\Lambda=1$.
\end{theorem}

\section{Preliminaries and non-linear projections in $\R^2$} 

\indent

This section is devoted to enumerate our tools to prove Theorem~\ref{tmain}. The results on the projections in $\R^2$ were previously studied by several authors, e.g. Hochman \cite{H1}, Hochman and Shemrkin \cite{HS}, Bond, \L aba and Zahl \cite{BLZ} etc. For the convenience of the reader, we state here these theorems and give short proofs.

First, we introduce some notations. Let $\Phi$ be an IFS on $\R^d$ with contracting similitudes in the form \eqref{edefIFS}. Denote the attractor of $\Phi$ by $\Lambda$. Let us denote the set of symbols by $\mathcal{S}=\left\{1,\dots,q\right\}$ and the symbolic space by $\Sigma=\mathcal{S}^{\mathbb{N}}$. Denote $\sigma$ the left-shift operator on $\Sigma$. Let us define the natural projection $\rho$ from $\Sigma$ to $\Lambda$ in the usual way, i.e. for any $\ii=(i_0,i_1,\dots)\in\Sigma$
$$
\rho(\ii)=\lim_{n\rightarrow\infty}f_{i_0}\circ f_{i_1}\circ\cdots\circ f_{i_n}(\underline{0}),
$$
where $\underline{0}=(0,\dots,0)\in\R^d$. It is easy to see that $\rho(\ii)=f_{i_0}(\rho(\sigma\ii))$.

Let $\underline{p}=(p_1,\dots,p_q)$ be a probability vector with strictly positive elements. Denote the Bernoulli measure on $\Sigma$ by $\nu=\underline{p}^{\mathbb{N}}$, then $\nu$ is left-shift invariant and ergodic. Then the measure $\mu=\rho_*\nu=\nu\circ\rho^{-1}$ is the unique self-similar measure with $\spt\mu=\Lambda$ and $\mu=\sum_{i=1}^qp_i(f_i)_*\mu$.

Let us denote the finite length words of symbols $\mathcal{S}$ by $\Sigma^*=\bigcup_{n=0}^{\infty}\mathcal{S}^n$. For an $\overline{\imath}=(i_0,\dots,i_{n-1})\in\Sigma^*$, denote $|\il|$ the length of $\il$ and for any $\il,\jl\in\Sigma^*$, denote the juxtaposition $\il\jl$ the finite length word $(i_0,\dots,i_{|\il|-1},j_0,\dots,j_{|\jl|-1})$.

For the composition of functions $f_{i_0}\circ\cdots\circ f_{i_{n-1}}$, we write $f_{\il}$, where $\il=(i_0,\dots,i_{n-1})$. We denote the fixed point of a function $f_{\il}$ by $\mathrm{Fix}(f_{\il})$. Denote $[\il]$ the cylinder set formed by $\il$,
$$[\il]:=\left\{\jj=(j_0,j_1,\dots)\in\Sigma:i_0=j_0,\dots,i_{|\il|-1}=j_{|\il|-1}\right\}.$$
We denote the projection of a cylinder set by $\Lambda_{\il}=\rho([\il])=f_{\il}(\Lambda)$, and we call it as a cylinder set of $\Lambda$. We note that if $\mu$ is a HSS measure (or $\Lambda$ is a HSS set) with IFS~$\Phi=\left\{f_i(\x)=\lambda_i\x+\underline{t}_i\right\}_{i=1}^q$ in $\R^d$ with SSC then for any $\il\in\Sigma^*$ the measure $\mu_{\il}:=\frac{\left.\mu\right|_{\Lambda_{\il}}}{\mu(\Lambda_{\il})}$ (or respectively $\Lambda_{\il}$) is also a self-similar measure (or self-similar set) with IFS~$\Phi_{\il}:=\left\{\lambda_i\x+f_{\il}(t_i)\right\}_{i=1}^q$. On the other hand, for any $\pi\in\Pi_{d,k}$ the measure $\pi\mu=\mu\circ\pi^{-1}$ is HSS measure (or respectively $\pi\Lambda$ is a HSS set), as well, with IFS~$\pi\Phi:=\left\{\lambda x+\pi(\underline{t}_i)\right\}_{i=1}^q$. We denote the $n$th iteration of the IFS by $\Phi^n=\left\{f_{\il}\right\}_{\il\in\mathcal{S}^n}$.

Our first approach of the study of homothetic self-similar sets is to find proper approximating subsystem.

\begin{prop}\label{psubsystem}
	Let $\Lambda$ be an HSS set in $\R^d$ with IFS $\Phi=\left\{f_i(\x)=\lambda_i\x+\underline{t}_i\right\}_{i=1}^q$. For every $\varepsilon>0$, there exists an IFS $\Phi'$ of the form $\left\{g_j(\x)=\lambda\underline{x}+\underline{t}_j'\right\}_{j=1}^{q'}$ with $\lambda\in(0,1)$ such that the attractor $\Lambda'$ of $\Phi'$ satisfies the SSC, $\Lambda'\subseteq\Lambda$ and $\dim_H\Lambda'>\dim_H\Lambda-\varepsilon$. Moreover, the functions of $\Phi'$ can be written as the composition of functions in $\Phi$.
	
	We call the attractor and self-similar measures of such a system $\Phi'$ as homogeneous homothetic self-similar set and measures (HHSS).
\end{prop}

The proof is analogous to the proof of Peres and Shmerkin \cite[Proposition~6]{PS}, therefore we omit it.

Let us denote the Hausdorff dimension of a measure $\mu$ by $\dim_H\mu$. That is,
$$
\dim_H\mu=\inf\left\{\dim_HA:\mu(A)>0\right\}.
$$
Let us define the upper and lower local dimension of a measure $\mu$ at a point $\x$ in the usual way by
$$
\underline{d}_{\mu}(\x)=\liminf_{r\rightarrow0+}\frac{\log\mu(B_r(\x))}{\log r}\text{ and }\overline{d}_{\mu}(\x)=\limsup_{r\rightarrow0+}\frac{\log\mu(B_r(\x))}{\log r},
$$
where $B_r(\x)$ is the ball with radius $r$ centered at $\x$. By \cite[Theorem~1.2]{FLR},
\begin{equation}\label{edimloc}
\dim_H\mu=\mu-\essinf_{\x}\underline{d}_{\mu}(\x).
\end{equation}
We say that the measure $\mu$ is {\em exact dimensional} if $\underline{d}_{\mu}(\x)=\overline{d}_{\mu}(\x)$ for $\mu$-a.e. $\x$. By \cite[Corollary~2.1]{FLR}, if $\mu$ is exact dimensional then
$$
\dim_H\mu=\inf\left\{\dim_HA:\mu(A)=1\right\}.
$$

\begin{lemma}\label{lac}
	Let $\mu$ and $\nu$ be Borel probability measures such that $\mu\ll\nu$ (that is, $\mu$ is absolutely continuous with respect to $\nu$) and $\nu$ is exact dimensional. Then $\dim_H\mu=\dim_H\nu$.
\end{lemma}

\begin{proof}
	Since $\mu\ll\nu$, for any measurable set $A$, if $\mu(A)>0$ then $\nu(A)>0$. Thus, $\dim_H\mu\geq\dim_H\nu$. On the other hand, since $\nu$ is exact dimensional
	\begin{multline*}
	\dim_H\nu=\inf\left\{\dim_HA:\nu(A)=1\right\}=\inf\left\{\dim_HA:\nu(A^c)=0\right\}\geq\\
	\inf\left\{\dim_HA:\mu(A^c)=0\right\}=\inf\left\{\dim_HA:\mu(A)=1\right\}\geq\dim_H\mu,
	\end{multline*}
	where $A^c$ denotes the complement of $A$.
\end{proof}

Our second approach is to approximate the non-linear projections of HSS measures with SSC by orthogonal projections. Let $g:\R^d\mapsto\R$ be a $C^1$ function. We denote the projection of a Borel measure $\mu$ on $\R^d$ by $g_*\mu=\mu\circ g^{-1}$. Let us denote the gradient of $g$ at a point $\x=(x_1,\dots,x_d)$ by $\nabla_{\x}g$, i.e.
$$
\nabla_{\x}g=\left(\begin{array}{c}
g'_{x_1}(\underline{x}) \\
\vdots   \\
g'_{x_d}(\underline{x})
\end{array}\right).
$$
Denote $\pi_{g,\x}\in\Pi_{d,1}$ the orthogonal projection from $\R^d$ to the subspace spanned by $\nabla_{\x}g$, that is, $\pi_{g,\x}(\underline{y})~=~\frac{<\nabla_{\x}g,\underline{y}>}{\|\nabla_{\x}g\|}$, where $<.,.>$ denotes the standard scalar product on $\R^d$ and $\|.\|$ denotes the induced norm. The next theorem is a consequence of the results of Hochman \cite{H1}.

\begin{theorem}\label{tHochman}
	Let $\mu$ be an HSS measure with SSC in $\R^d$ and let $g:\R^d\mapsto\R$ be a $C^1$ function with $\|\nabla_{\underline{x}}g\|\neq0$ for every $\underline{x}\in\mathrm{spt}\mu$. Then
	\begin{equation*} 
	\dim_Hg_*\mu\geq\mu-\essinf_{\underline{x}}\dim_H\pi_{g,\x}\mu.
	\end{equation*}
\end{theorem}

\begin{proof}
	Let $\mu$ be an HSS measure with SSC in $\R^d$. Then by \cite[Example~4.3]{H1} the measure $\mu$ is a \textit{homogeneous uniformly scaling measure}, see \cite[Definition~1.5(3) and Defintion~1.35]{H1}. Let $P$ be the ergodic fractal distribution generated by $\mu$, see \cite[Definition~1.2, Definition~1.5(1) and Proposition~1.36]{H1}. For a $\pi\in\Pi_{d,1}$, let $$E_P(\pi)=\int\dim_H\pi\nu dP(\nu).$$
	Applying \cite[Theorem~1.23]{H1} and \cite[Proposition~1.36]{H1} we have for any $g:\R^d\mapsto\R$ $C^1$ function with $\|\nabla_{\underline{x}}g\|\neq0$
	$$
	\dim_Hg_*\mu\geq\mu-\essinf_{\underline{x}}E_P(\pi_{g,\x}).
	$$
	By \cite[Proposition~1.36]{H1} for $P$-a.e $\nu$ measure there exists a ball $B$ that $\mu\ll (T_B)_*\nu$, where $T_{B_r(\x)}(\underline{y})=\frac{\underline{y}-\x}{r}$. Hence, $\pi\mu\ll \pi(T_B)_*\nu$ for every $\pi\in\Pi_{d,1}$ and $P$-a.e. $\nu$. On the other hand, by \cite[Theorem~1.22]{H1} the measure $\pi\nu$ is exact dimensional for $P$-a.e. $\nu$. Since $T_B$ is a bi-Lipschitz map, by Lemma~\ref{lac}, $\dim_H\pi\mu=\dim_H\pi\nu$ for every $\pi\in\Pi_{d,1}$ and $P$-a.e. $\nu$, which implies that $E_P(\pi)=\dim_H\pi\mu$.
\end{proof}

As a consequence of Theorem~\ref{tHochman} and \cite[Theorem~1.8]{H2}, we state here a modified version of the proposition of Hochman, published in Bond, \L aba and Zahl \cite[Proposition~2.6]{BLZ}.

\begin{prop}\label{pBLZ}
	Let $\mu$ be a HHSS measure with SSC in $\R^2$ such that $\spt\mu$ is not contained in any line. Suppose that $g:\R^2\mapsto\R$ $C^2$ map such that $\|\nabla_{\x}g\|\neq0$ and
	$$\left\|\left(\begin{array}{c}
	g_{xx}''(\x)g_y'(\x)-g_{xy}''(\x)g_x'(\x) \\
	g_{xy}''(\x)g_y'(\x)-g_{yy}''(\x)g_x'(\x)
	\end{array}\right)\right\|\neq0$$
	for every $\x\in\spt\mu$. Then
	$$\dim_Hg_*\mu=\min\left\{1,\dim_H\mu\right\}.$$
\end{prop}

Before we prove the proposition, we need a technical lemma.

\begin{lemma}\label{ltech1}
	Let $\mu$ be a HHSS measure with SSC in $\R^2$ such that $\spt\mu$ is not contained in any line. Then there exists a constant $c>0$ that $\dim_H\pi\mu\geq c>0$ for every $\pi\in\Pi_{2,1}$.
\end{lemma}

\begin{proof}
	Let $\mu$ be a HHSS measure with SSC in $\R^2$ such that $\spt\mu$ is not contained in any line and let $\left\{f_i(\x)=\lambda\x+\underline{t}_i\right\}_{i=1}^q$ the corresponding IFS and $\underline{p}=(p_1,\dots,p_q)$ the corresponding probability vector.
	
	Since $\spt\mu$ is not contained in any line, there exist three fixed points of the functions, let say $f_1,f_2$ and $f_3$, form a triangle. Let us denote the sides of the triangle by $a$, $b$ and $c$. Let $\kappa=\inf_{\pi\in\Pi_{2,1}}\max\left\{|\pi a|,|\pi b|,|\pi c|\right\}>~0$ and let $N=\lceil\frac{\log\kappa/(4|\spt\mu|)}{\log\lambda}\rceil$, where $|.|$ denotes the diameter of a set. Let $x\in\pi\spt\mu$ be arbitrary, and let $$z_n(x):=\sum_{\substack{\il\in\mathcal{S}^{nN} \\ B_{\frac{\kappa}{4}\lambda^{nN}}(x)\cap\pi\Lambda_{\il}\neq\emptyset}}\nu([\il]).$$
		
	It is easy to see by the definition of $N$ and $\kappa$ that there exists an $\il\in\Sigma^*$ with $|\il|=N$ that $B_{\frac{\kappa}{4}\lambda^{N}}(x)\cap\pi\Lambda_{\il}=\emptyset$. Thus, $z_1(x)\leq(1-p_{\min}^N)$, where $p_{\min}=\min\left\{p_1,\dots,p_q\right\}$. On the other hand, for every $\il\in\Sigma^*$ with $|\il|=nN$ and $B_{\frac{\kappa}{4}\lambda^{nN}}(x)\cap\pi\Lambda_{\il}\neq\emptyset$ there exists a $\jl\in\Sigma$ with $|\jl|=N$ that $B_{\frac{\kappa}{4}\lambda^{(n+1)N}}(x)\cap\pi\Lambda_{\il\jl}=\emptyset$. Thus,
	\begin{equation}\label{einduct}
	\sum_{\substack{\jl\in\mathcal{S}^{N} \\ B_{\frac{\kappa}{4}\lambda^{(n+1)N}}(x)\cap\pi\Lambda_{\il\jl}\neq\emptyset}}\nu([\jl])\leq1-p_{\min}^N.
	\end{equation}
	
	Now we prove by induction that $z_n(x)\leq(1-p_{\min}^N)^n$. For $n=1$ it has already been showed. Assume that it holds for $n$. Then by \eqref{einduct}
	\begin{multline*}
		z_{n+1}(x)=\sum_{\substack{\il\in\mathcal{S}^{(n+1)N} \\ B_{\frac{\kappa}{4}\lambda^{(n+1)N}}(x)\cap\pi\Lambda_{\il}\neq\emptyset}}\nu([\il])=\sum_{\substack{\il\in\mathcal{S}^{nN} \\ B_{\frac{\kappa}{4}\lambda^{nN}}(x)\cap\pi\Lambda_{\il}\neq\emptyset}}\sum_{\substack{\jl\in\mathcal{S}^{N} \\ B_{\frac{\kappa}{4}\lambda^{(n+1)N}}(x)\cap\pi\Lambda_{\il\jl}\neq\emptyset}}\nu([\il\jl])=\\
		\sum_{\substack{\il\in\mathcal{S}^{nN} \\ B_{\frac{\kappa}{4}\lambda^{nN}}(x)\cap\pi\Lambda_{\il}\neq\emptyset}}\nu([\il])\sum_{\substack{\jl\in\mathcal{S}^{N} \\ B_{\frac{\kappa}{4}\lambda^{(n+1)N}}(x)\cap\pi\Lambda_{\il\jl}\neq\emptyset}}\nu([\jl])\leq(1-p_{\min}^N)z_n(x)\leq(1-p_{\min}^N)^{n+1}.
	\end{multline*}
	Hence, for any $x\in\spt\mu$
	$$\liminf_{r\rightarrow0+}\frac{\log\mu(B_r(x))}{\log r}=\liminf_{n\rightarrow\infty}\frac{\log\mu(B_{\frac{\kappa}{4}\lambda^{nN}}(x))}{nN\log\lambda}\geq\liminf_{n\rightarrow\infty}\frac{\log z_n(x)}{nN\log\lambda}=\frac{\log(1-p_{\min}^N)}{N\log\lambda}>0.$$
	Which implies by \eqref{edimloc} that $\dim_H\pi\mu\geq\frac{\log(1-p_{\min}^N)}{N\log\lambda}>0$ for every $\pi\in\Pi_{2,1}$.
\end{proof}

\begin{proof}[Proof of Proposition~\ref{pBLZ}]
	Let $\mu$ be a HHSS measure with SSC such that $\spt\mu$ is not contained in any line. Since $\dim_Hg_*\mu\leq\min\left\{1,\dim_H\mu\right\}$, it is enough to show the lower bound. By Theorem~\ref{tHochman} we have
	$$\dim_Hg_*\mu\geq\mu-\essinf_{\x}\dim_H\pi_{g,\x}\mu.$$ Thus, it is enough to show that
	\begin{equation}\label{eenough1}
	\dim_H\pi_{g,\x}\mu=\min\left\{1,\dim_H\mu\right\}\text{ for $\mu$-a.e. $\x$.}
	\end{equation}
	
	If $\mu$ is a HHSS measure with IFS $\left\{\lambda\x+\underline{t}_i\right\}_{i=1}^q$ then for any $\pi\in\Pi_{2,1}$ the measure $\pi\mu$ is HHSS measure, as well, with IFS $\left\{\lambda x+\pi(\underline{t}_i)\right\}_{i=1}^q$. By using the parametrization $\pi_{\theta}(\x)=<(\cos\theta,\sin\theta),\x>$ and \cite[Theorem~1.8]{H2}, it follows that
	$$\dim_P\left\{\theta\in[0,\pi):\dim_H\pi_{\theta}\mu<\min\left\{1,\dim_H\mu\right\}\right\}=0.$$
	Hence, to verify \eqref{eenough1} it is enough to show that
	\begin{equation*} 
	\dim_Hf_*\mu>0,
	\end{equation*}
	where $f(\x)=\arctan\left(\frac{g'_{x_1}(\x)}{g_{x_2}'(\x)}\right)$.  By our assumption $\|\nabla_{\x}f\|\neq0$ for every $\x\in\spt\mu$. By applying Theorem~\ref{tHochman} and Lemma~\ref{ltech1}, we get
	\[
	\dim_Hf_*\mu\geq\mu-\essinf_{\x}\dim_H\pi_{f,\x}\mu\geq\inf_{\pi\in\Pi_{2,1}}\dim_H\pi\mu\geq c>0.
	\]
\end{proof}

As a consequence of Proposition~\ref{pBLZ}, we state here the analogue of \cite[Proposition~2.5]{BLZ} but for measures, which plays important role for the further studies.

\begin{cor}\label{cradproj}
	If $\mu$ is a HHSS measure with SSC in $\R^2$ such that $\underline{0}\notin\spt\mu$ and $\spt\mu$ is not contained in any line then \begin{equation}\label{eeleg}
	\dim_H(P_2)_*\mu=\min\left\{1,\dim_H\mu\right\}.
	\end{equation}
\end{cor}

\begin{proof}
	Since $\mu$ can be written as a convex combination of self-similar measures restricted to cylinder sets, we have
	$$\dim_H(P_2)_*\mu=\min_{\il\in\mathcal{S}^n}\dim_H(P_2)_*\mu_{\il}$$
	for every $n\geq1$. Thus it is enough to show that for sufficiently large $n\geq1$ \eqref{eeleg} holds for any $\il\in\mathcal{S}^n$. By choosing $n$ sufficiently large and by applying a rotation transformation, without loss of generality we may assume that $\spt\mu_{\il}$ is contained in the upper half plane and it is separated away from the $x$-axis.
	
	Since the map $h:x\mapsto(x,\sqrt{1-x^2})$ is bi-Lipschitz for every $x\in(-1+\varepsilon,1+\varepsilon)$, it is enough to show that for the map $g:(x,y)\mapsto\frac{x}{\sqrt{x^2+y^2}}$, $\dim_Hg_*\mu_{\il}=\min\left\{1,\dim_H\mu_{\il}\right\}.$ Indeed, $g$ satisfies the assumptions of Proposition~\ref{pBLZ}.
\end{proof}

As another consequence of Proposition~\ref{pBLZ} we can state the following theorem for general self-similar sets in $\R^2$.

\begin{theorem}\label{tBLZ}
	Let $\Lambda$ be an arbitrary self-similar set in $\R^2$ not contained in any line. Suppose that $g:\R^2\mapsto\R$ is a $C^2$ map such that $\|\nabla_{\x}g\|\neq0$ and
	$$\left\|\left(\begin{array}{c}
	g_{xx}''(\x)g_y'(\x)-g_{xy}''(\x)g_x'(\x) \\
	g_{xy}''(\x)g_y'(\x)-g_{yy}''(\x)g_x'(\x)
	\end{array}\right)\right\|\neq0$$
	for every $\x\in\Lambda$. Then
	$$\dim_Hg(\Lambda)=\min\left\{1,\dim_H\Lambda\right\}.$$
\end{theorem}

\begin{proof}
	Let $\Lambda$ be a self-similar set in $\R^2$ not contained in any line. 
	Applying \cite[Lemma~3.4]{O}, for every $\varepsilon>0$ there exists a self-similar set $\Lambda'\subseteq\Lambda$ not contained in any line such that $\dim_H\Lambda'\geq\dim_H\Lambda-\varepsilon$ and its the attractor of IFS $\Phi'$ satisfying SSC. If one of the functions of $\Phi'$ contains an irrational rotation then by \cite[Corollary~1.7]{HS}
	$$\dim_Hg(\Lambda)\geq\dim_Hg(\Lambda')=\min\left\{1,\dim_H\Lambda'\right\}\geq\min\left\{1,\dim_H\Lambda\right\}-\varepsilon.$$
	
	If none of the functions of $\Phi'$ contains irrational rotation then by \cite[Lemma~4.2]{O} there exists a self-similar set $\Lambda''\subseteq\Lambda$ such that
	$\dim_H\Lambda''\geq\dim_H\Lambda-2\varepsilon$ and the similitudes of generating IFS~$\Phi''$ of $\Lambda''$ do not contain any rotation or reflection, i.e. it is a HSS set with SSC. By Proposition~\ref{psubsystem}, there exists a HHSS set $\Lambda'''$ with SSC that $\Lambda'''\subseteq\Lambda$ and $\dim_H\Lambda'''\geq\dim_H\Lambda-3\varepsilon$.
	
	Let $\mu$ be the natural self-similar measure on $\Lambda'''$, that is, $\mu$ is the equidistributed self-similar measure on the cylinder sets. Hence, $\dim_H\mu=\dim_H\Lambda'''$. By Proposition~\ref{pBLZ},
	$$\dim_Hg(\Lambda)\geq\dim_Hg_*\mu=\min\left\{1,\dim_H\mu\right\}=\min\left\{1,\dim_H\Lambda'''\right\}\geq\min\left\{1,\dim_H\Lambda\right\}-3\varepsilon.$$
	Since $\varepsilon>0$ was arbitrary, the statement of the theorem is proven.
\end{proof}

As a corollary of Theorem~\ref{tBLZ}, one can prove a weaker version of Falconer's distance set conjecture in $\R^2$. This is just a little bit stronger than Orponen's result \cite[Theorem~1.2]{O}, since we only assume that $\dim_H\Lambda\geq1$ and we do not need that $\mathcal{H}^1(\Lambda)>0$.

\begin{cor}\label{cdist}
	If $\Lambda$ is a self-similar set in $\R^2$ with $\dim_H\Lambda\geq1$. Then
	$$\dim_HD(\Lambda)=1,$$
	where $D(\Lambda)$ denotes the distance set of $\Lambda$ defined in \eqref{edistset}.
\end{cor}

\begin{proof}
	If $\Lambda$ is contained in a line then $\dim_HD(\Lambda)=\dim_H\Lambda$. So, we may assume that $\Lambda$ is not contained in any line. Let $\underline{a}$ be an arbitrary element of $\Lambda$ and let $\Lambda_{\il}$ be a cylinder set such that $\mathrm{dist}(\underline{a},\Lambda_{\il})>0$. Then $D_{\underline{a}}(\x)=\|\x-\underline{a}\|$ satisfies the conditions of Theorem~\ref{tBLZ} with self-similar set $\Lambda_{\il}$. Thus,
	$$\dim_HD(\Lambda)\geq\dim_HD_{\underline{a}}(\Lambda)\geq\dim_HD_{\underline{a}}(\Lambda_{\il})=\min\left\{1,\dim_H\Lambda_{\il}\right\}=\min\left\{1,\dim_H\Lambda\right\}=1.
	$$
\end{proof}

Another corollary of Theorem~\ref{tBLZ} is \eqref{eprod2}.

\begin{cor}\label{cprod2}
	If $\Lambda$ is a self-similar set in $\R$ then $\dim_H\Lambda\cdot\Lambda=\min\left\{2\dim_H\Lambda,1\right\}$.
\end{cor}

\begin{proof}
	Let $\Lambda$ be an arbitrary self-similar set on $\R$. Without loss of generality, we may assume that $\Lambda$ is not a singleton. Then there exists a cylinder set $\Lambda_{\il}$ of $\Lambda$ that every element in $\Lambda_{\il}$ is either strictly positive or strictly negative.
	
	By \cite[Proposition~6]{PS}, for every $\varepsilon>0$ there exists a self-similar set $\Lambda'\subseteq\Lambda_{\il}$ such that $\dim_H\Lambda'\geq\dim_H\Lambda-\varepsilon$ and its the attractor of IFS $\Phi$ satisfying SSC and has the form
	$$
	\Phi=\left\{f_i(x)=\lambda x+t_i\right\}_{i=1}^{q}.
	$$
	Then $\Lambda'\times\Lambda'$ is a self-similar set with SSC in $\R^2$ with IFS $$\Phi'=\left\{h_i(\underline{x})=\lambda\x+(t_i,t_j)\right\}_{i,j=1}^q.$$
	
	Let $g(x,y)=xy$. Then
	$$\|\nabla_{\x}g\|=\sqrt{y^2+x^2}\neq0\text{ and }\left\|\left(\begin{array}{c}
	g_{xx}''(\x)g_y'(\x)-g_{xy}''(\x)g_x'(\x) \\
	g_{xy}''(\x)g_y'(\x)-g_{yy}''(\x)g_x'(\x)
	\end{array}\right)\right\|=\sqrt{x^2+y^2}\neq0$$ for any $(x,y)\in\R^2\backslash\left\{(0,0)\right\}$. Thus by Theorem~\ref{tBLZ}
	$$\dim_H\Lambda\cdot\Lambda\geq\dim_H\Lambda'\cdot\Lambda'=\dim_Hg(\Lambda'\times\Lambda')=\min\left\{1,\dim_H\Lambda'\times\Lambda'\right\}\geq\min\left\{1,2\dim_H\Lambda\right\}-2\varepsilon,$$
	where we used that $\dim_H\Lambda'\times\Lambda'=2\dim_H\Lambda'$, see \cite[Corollary~7.4]{F}. Since $\varepsilon>0$ was arbitrary, the proof is complete.
\end{proof}

Similarly to the proof of Proposition~\ref{pBLZ}, to prove our main Theorem~\ref{tmain}, we need an upper bound for the exceptional directions for the orthogonal projections in $\Pi_{3,1}$. For a vector $\n\in S^{d-1}$ let $\pi_{\n}\in\Pi_{d,1}$ be the orthogonal projection to the subspace generated by $\n$, i.e. $\pi_{\n}(\x)=<\x,\n>$.

\begin{prop}\label{texcept}
	Let $\mu$ be a HHSS measure in $\R^3$ with SSC. Then \begin{equation}\label{eexcept}
	\dim_P\left\{\n\in S^2:\dim_H\pi_{\n}\mu<\min\left\{1,\dim_H\mu\right\}\right\}\leq1.
	\end{equation}
\end{prop}

Proposition~\ref{texcept} follows from Hochman~\cite[Theorem~1.10]{H3}.

Finally, we state here the dimension conservation phenomena for HSS measures, first showed by Furstenberg \cite{Fu} and generalized by Falconer and Jin \cite{FJ}.

\begin{theorem}\label{tdincon&lower}
	Let $\mu$ be an HSS measure with SSC in $\R^d$ and let $\pi\in\Pi_{d,k}$ be arbitrary. Then
	\begin{equation}\label{edimc}
	\dim_H\pi\mu+\dim_H\mu_{\pi^{-1}(\x)}=\dim_H\mu\text{ for $\pi\mu$-a.e. $\x\in\R^k$,}
	\end{equation}
	where $\mu_{\pi^{-1}(\x)}$ denote the conditional measures of $\mu$ on the fibres $\pi^{-1}(\x)$. Moreover,
	\begin{equation}\label{elowersemi}
	\pi\mapsto\dim_H\pi\mu\text{ is lower semi-continuous.}
	\end{equation}
\end{theorem}

For the proof of the theorem we refer to Hochman \cite[Theorem~1.37]{H1}. 

\section{Radial projection in $\R^3$}

\indent

The critical point of our study is the examination of the radial projection. Unfortunately, we cannot prove the analogue of Corollary~\ref{cradproj} in general. However, we are able to show that if an HSS set has dimension strictly larger than $1$ then there exists a HHSS measure such that its support is contained in the HSS set, and its radial projection has dimension strictly larger than $1$.

\begin{theorem}\label{tradproj3}
	Let $\Lambda$ be a HSS set in $\R^3$ such that $\dim_H\Lambda>1$ and $\Lambda$ is not contained in any plane. Then there exists a $\mu$ HHSS measure such that $\mathrm{spt}\mu\subseteq\Lambda$ and $\dim_H\mu\geq\dim_H(P_3)_*\mu>1$.
\end{theorem}

Let us denote the closed double cone with vertex $\x\in\R^3$, angle $\alpha$, and axis $\underline{v}\in\R^3$ with $\|\vv\|=1$ by $C_{\alpha,\underline{v}}(\x)$. That is,
$$C_{\alpha,\underline{v}}(\x)=\left\{\underline{y}\in\R^3:|<\x-\underline{y},\underline{v}>|\geq|\cos(\alpha)|\|\x-\y\|\right\}.$$
In other words, the angle between $\x-\y$ and $\vv$ is less than or equal to $\alpha$. First, we show the following lemma.

\begin{lemma}\label{lnotinplane}
	Let $\Lambda$ be a HHSS set in $\R^3$ such that it is not contained in any plane. Then for every vector $\underline{v}\in\R^3$ with $\|\vv\|=1$ and $\x\in\Lambda$ there exists an $\pi/2>\alpha>0$ such that for every $r>0$
	$$\mathrm{int}(B_r(\x)\cap C_{\alpha,\vv}(\x))\cap\Lambda\neq\emptyset,$$
	where $\mathrm{int}(A)$ denotes the interior of a set $A$.
\end{lemma}

\begin{proof}
	We argue by contradiction. Assume that there exist vector $\underline{v}\in\R^3$ with $\|\vv\|=1$ and $\x\in\Lambda$ such that for every $\pi/2>\alpha>0$ there exists an $r=r(\alpha)>0$ that $$\mathrm{int}(B_r(\x)\cap C_{\alpha,\vv}(\x))\cap\Lambda=\emptyset.$$ Let $\Phi=\left\{f_i(\x)=\lambda\x+\underline{t}_i\right\}_{i=1}^q$ be the corresponding IFS and let $n(r)=\min\left\{n:\lambda^n<r\right\}$.
	
	For a $\pi/2>\alpha>0$ if $\x\in\Lambda_{\il}$ with $|\il|\geq n(r(\alpha))$ then $\Lambda_{\il}\subseteq \mathrm{int}(B_{r(\alpha)}(\x))\cap\Lambda$. Thus by our assumption $\Lambda_{\il}\cap \mathrm{int}(C_{\alpha,\vv}(\x))=\emptyset.$ Since $\Phi$ does not contain any orthogonal transformation. $$\Lambda\cap\mathrm{int}(C_{\alpha,\vv}(f_{\il}^{-1}(\x)))=f_{\il}^{-1}(\Lambda_{\il}\cap\mathrm{int}(C_{\alpha,\vv}(\x)))=\emptyset.$$
	Thus, for every $\pi/2>\alpha>0$ there exists a $N\geq1$ such that for every $\il$ with $\x=\rho(\ii)=f_{\il}(\rho(\sigma^{|\il|}\ii))$ and $|\il|\geq N$
	$$\Lambda\cap\mathrm{int}(C_{\alpha,\vv}(f_{\il}^{-1}(\x)))=\emptyset.$$
	Let $\y$ be a density point of the sequence $\left\{f_{\il}^{-1}(\x)\right\}$. Since $\Lambda$ is compact, $\y\in\Lambda$ and $\Lambda\cap \mathrm{int}C_{\alpha,\vv}(\y)=\emptyset.$ But $\alpha$ was arbitrary, thus $\Lambda$ must be contained in a plane with normal vector $\vv$ and containing $\y$ which is a contradiction.	
\end{proof}

Denote by $V_{\pi}$ the subspace to which $\pi\in\Pi_{3,2}$ projects and denote the normal vector of $V_{\pi}$ by $\underline{n}_{\pi}\in\R^3$ with $\|\underline{n}_{\pi}\|=1$. For a projection $\pi\in\Pi_{3,2}$, let
\begin{equation}\label{eepszilon}
\sin(\varepsilon_{\pi})=\inf_{\x\neq\underline{y}\in\Lambda}\frac{\|\underline{n}_{\pi}\times(\x-\underline{y})\|}{\|\x-\underline{y}\|}.
\end{equation}
Since $\Lambda$ is compact, if $\pi\Lambda$ satisfies the SSC then $\sin(\varepsilon_{\pi})>0$ . Let $\Gamma_{\pi}$ be as follows
$$\Gamma_{\pi}=\left\{(\x,\y)\in\Lambda\times\Lambda:\x\neq\y\ \&\ \sin(\varepsilon_{\pi})=\frac{\|\underline{n}_{\pi}\times(\x-\underline{y})\|}{\|\x-\underline{y}\|}\right\}.$$
Since $\Phi=\{f_i\}_{i=1}^q$ is orthogonal transformation free,
$$\sin(\varepsilon_{\pi})=\min_{i\neq j}\inf_{\x\in f_i(\Lambda),\underline{y}\in f_j(\Lambda)}\frac{\|\underline{n}_{\pi}\times(\x-\underline{y})\|}{\|\x-\underline{y}\|}.$$
Moreover, by compactness, there are $i\neq j$, $\x_i\in f_i(\Lambda)$, $\y_j\in f_j(\Lambda)$ such that $(\x_i,\y_j)\in\Gamma_{\pi}\neq\emptyset$.
By definition
\begin{equation}\label{eempty}
\mathrm{int}\left(C_{\varepsilon_{\pi},\n_{\pi}}(\x)\right)\cap\Lambda=\emptyset\text{ for every $\x\in\Lambda$.}
\end{equation}	

\begin{figure}
	\centering
	\includegraphics[width=0.4\linewidth]{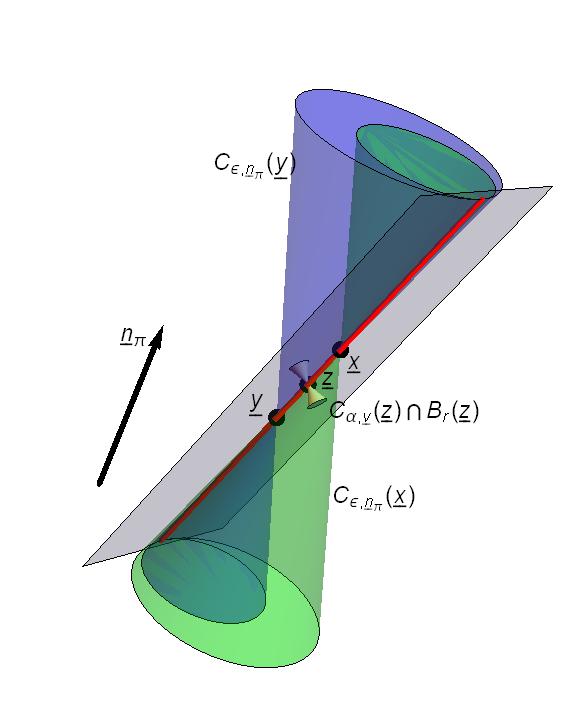}
	\caption{Cones and points for $(\x,\y),(\x,\z),(\y,\z)\in\Gamma_{\pi}$.}\label{fig:conepic}
\end{figure}
	
\begin{lemma}\label{lproj2to1hoz}
	Let $\Lambda$ be a HHSS set not contained in any plane and suppose that $\pi\Lambda$ satisfies the SSC. If $(\x,\y),(\x,\z)\in\Gamma_{\pi}$ then $(\y,\z)\notin\Gamma_{\pi}$. Thus, if $(\x,\y)\in\Gamma_{\pi}$ then $l_{\x,\y}\cap\Lambda\backslash\{\x,\y\}=\emptyset$, where $l_{\x,\y}$ is the line containing $\x$ and $\y$.
\end{lemma}

\begin{proof}
	Let us suppose that $(\x,\y),(\x,\z),(\y,\z)\in\Gamma_{\pi}$. It is easy to see that $\x,\y,\z$ must be contained in one line. Indeed, $\z$ must be a common element of the boundary of the cones $C_{\varepsilon_{\pi},\n_{\pi}}(\x)$ and $C_{\varepsilon_{\pi},\n_{\pi}}(\y)$, see Figure~\ref{fig:conepic}. Without loss of generality, assume that $\z$ is between $\x$ and $\y$. Let $V$ be the common tangent plane of the cones $C_{\varepsilon_{\pi},\n_{\pi}}(\x)$ and $C_{\varepsilon_{\pi},\n_{\pi}}(\y)$, and let $\vv$ be its normal vector. Applying Lemma~\ref{lnotinplane}, there exists an $\pi/2>\alpha>0$ that $\mathrm{int}(B_r(\z)\cap C_{\alpha,\vv}(\z))\cap\Lambda\neq\emptyset,$ for every $r>0$. Let $r>0$ be sufficiently small that $\mathrm{int}(B_r(\z)\cap C_{\alpha,\vv}(\z))\subseteq\mathrm{int}(C_{\varepsilon_{\pi},\n_{\pi}}(\x)\cap C_{\varepsilon_{\pi},\n_{\pi}}(\y))$. Then
	$$\emptyset\neq\mathrm{int}(B_r(\z)\cap C_{\alpha,\vv}(\z))\cap\Lambda\subseteq\mathrm{int}(C_{\varepsilon_{\pi},\n_{\pi}}(\x))\cap\Lambda.$$
	But by \eqref{eempty}, $\mathrm{int}(C_{\varepsilon_{\pi},\n_{\pi}}(\x))\cap\Lambda=\emptyset$, which is a contradiction.
\end{proof}

\begin{prop}\label{pgoodmeasure}
	Let $\Lambda$ be a HSS set in $\R^3$ such that $\dim_H\Lambda>1$ and it is not contained in any plane. Then there exists an orthogonal projection $\pi\in\Pi_{3,2}$ and a self-similar measure $\mu$ such that $\spt\mu\subseteq\Lambda$, $\spt\mu$ is not contained in any plane, $\pi\mu$ satisfies the SSC (but not w.r.t the IFS generating $\mu$), and $\dim_H\mu>\dim_H\pi\mu>1$.
\end{prop}

\begin{proof}
	Let $\Lambda$ be a HSS set satisfying the assumptions. By Marstrand's projection theorem \cite[Corollary~9.4, Corollary~9.8]{M} there exists a $\pi_1\in\Pi_{3,2}$ such that $\dim_H\pi_1\Lambda=\min\left\{2,\dim_H\Lambda\right\}$. The set $\pi_1\Lambda$ is a HSS set in $\R^2$. Applying Proposition~\ref{psubsystem}, there exists a HHSS set $\Lambda^1$ and an IFS $\Phi_1=\left\{f_i(\x)=\lambda\x+\underline{t}_i\right\}_{i=1}^q$ with SSC such that $\Lambda^1\subseteq\Lambda$, $\pi_1\Lambda^1$ satisfies the SSC and $2>\dim_H\Lambda^1=\dim_H\pi_1\Lambda^1>1$.
	
	Let $\varepsilon_{\pi_1}$ be defined in \eqref{eepszilon}. By compactness, there are $\x_i\in f_i(\Lambda^1)$, $\y_j\in f_j(\Lambda^1)$ such that $i\neq j$ and $(\x_i,\y_j)\in\Gamma_{\pi_1}$. Let us fix such $i\neq j$ and $\x_i,\y_j$. Denote the projection onto the subspace with normal vector $\frac{\x_i-\y_j}{\|\x_i-\y_j\|}$ by $\pi_2$. Then by Lemma~\ref{lproj2to1hoz}, the projection $\pi_2$ is 2 to 1 on $\Lambda^1$. Thus, by \cite[Corollary~4.16]{FH}, $\dim_H\Lambda^1=\dim_H\pi_2\Lambda^1$, but clearly, the SSC does not hold.
	
	Let $\delta>0$ be sufficiently small such that $\dim_H\Lambda^1-3\delta>1$. Let us fix $\il,\jl\in\Sigma^*$ such that $|\il|=|\jl|$, $\x_i\in f_{\il}(\Lambda^1)$, $\y_j\in f_{\jl}(\Lambda^1)$ and choose $m:=|\il|=|\jl|$ sufficiently large that the attractor $\widetilde{\Lambda}$ of the IFS~$\widetilde{\Phi}:=\left\{f_{i_0}\circ\cdots\circ f_{i_{m-1}}\right\}_{i_0,\dots,i_{m-1}=1}^q\backslash\left\{f_{\il},f_{\jl}\right\}$ satisfies $\dim_H\widetilde{\Lambda}\geq\dim_H\Lambda^1-\delta$. Since $\pi_2$ is still at most 2 to 1 on the smaller set $\widetilde{\Lambda}$, we have $\dim_H\widetilde{\Lambda}=\dim_H\pi_2\widetilde{\Lambda}$. Let us observe that $\pi_2\x_i=\pi_2\y_j\notin\pi_2\widetilde{\Lambda}$.
	
	Let $\widetilde{\mu}$ be the natural HSS measure on $\widetilde{\Lambda}$. By Theorem~\ref{tdincon&lower}\eqref{elowersemi} the function $\pi\mapsto\dim_H\pi\widetilde{\mu}$ is lower semi-continuous at $\pi_2$. Hence, $\pi\mapsto\dim_H\pi\widetilde{\Lambda}$ is lower semi-continuous at $\pi_2$. Let $\beta>0$ sufficiently small such that for every projection $\pi\in\Pi_{3,2}$, with $\|\n_{\pi}\times\n_{\pi_2}\|<|\sin(\beta)|$, $\pi\x_i,\pi\y_j\notin\pi\widetilde{\Lambda}$ and
	$$\dim_H\pi\widetilde{\Lambda}\geq\dim_H\pi_2\widetilde{\Lambda}-\delta=\dim_H\widetilde{\Lambda}-\delta.$$
	
Since the fixed points of the iterates of the functions are dense in $\widetilde{\Lambda}$, by compactness, we may find $\hbar_{1},\hbar_{2}\in\Sigma^*$ that $\pi(\mathrm{Fix}(f_{\il\hbar_1})),\pi(\mathrm{Fix}(f_{\jl\hbar_2}))\notin\pi\widetilde{\Lambda}$ for every projection $\pi\in\Pi_{3,2}$,  with $\|\n_{\pi}\times\n_{\pi_2}\|<|\sin(\beta)|$ and	
	$$\frac{\|(\mathrm{Fix}(f_{\il\hbar_1})-\mathrm{Fix}(f_{\jl\hbar_2}))\times(\x_i-\y_j)\|}{\|\mathrm{Fix}(f_{\il\hbar_1})-\mathrm{Fix}(f_{\jl\hbar_2})\|\|\x_i-\y_j\|}<|\sin(\beta)|.$$

	Denote the projection onto the subspace with normal vector $\frac{\mathrm{Fix}(f_{\il\hbar_1})-\mathrm{Fix}(f_{\jl\hbar_2})}{\|\mathrm{Fix}(f_{\il\hbar_1})-\mathrm{Fix}(f_{\jl\hbar_2})\|}$ by $\pi'$.  Applying Proposition~\ref{psubsystem} for $\pi'\widetilde{\Phi}$, there exist a HHSS set $\widetilde{\Lambda}^1$ and an IFS $\widetilde{\Phi}_1$ with SSC such that $\widetilde{\Lambda}^1\subseteq\widetilde{\Lambda}$, $\pi'\widetilde{\Lambda}^1$ satisfies the SSC and $\dim_H\widetilde{\Lambda}^1=\dim_H\pi'\widetilde{\Lambda}^1>\dim_H\pi'\widetilde{\Lambda}-\delta$.
	
	We claim that there exist $m,k\geq1$ that the IFS $\left(\pi'\widetilde{\Phi}_1\right)^m\cup\left\{\overbrace{\pi'f_{\il\hbar_1}\circ\cdots\circ\pi'f_{\il\hbar_1}}^k=:\pi'f_{\il\hbar_1}^k\right\}$ satisfies the SSC and it is homogeneous.
	
	Indeed, since the system $\pi'\widetilde{\Phi}_1$ satisfies SSC and is homogeneous, then for every $m\geq1$ $\left(\pi'\widetilde{\Phi}_1\right)^m$ still satisfies SSC and is homogeneous. By Proposition~\ref{psubsystem}, the contraction ratio of $\widetilde{\Phi}_1$ is $\lambda^l$ for an $l\geq1$. On the other hand, the contraction ratio of $f_{\il\hbar_1}$ is $\lambda^{|\il\hbar_1|}$. Now, let us fix the ratio $k/m=l/|\il\hbar_1|$. Since $\pi'(\mathrm{Fix}(f_{\il\hbar_1})),\pi'(\mathrm{Fix}(f_{\jl\hbar_2}))\notin\pi'\widetilde{\Lambda}$, by choosing $k$ sufficiently large, the SSC holds.
	
	Let $\Phi':=(\widetilde{\Phi}_1)^m\cup\left\{f_{\il\hbar_1}^k,f_{\il\hbar_2}^k\right\}$ and its attractor $\Lambda'$. Observe that $\pi'(\mathrm{Fix}(f_{\il\hbar_1}))=\pi'(\mathrm{Fix}(f_{\jl\hbar_2}))$. Thus, $\pi'f_{\il\hbar_1}\equiv\pi'f_{\il\hbar_2}$, i.e. there are exact overlaps. Hence, $\pi'\Phi'=\left(\pi'\widetilde{\Phi}_1\right)^m\cup\left\{\pi'f_{\il\hbar_1}^k\right\}$ and therefore, satisfies SSC.
	
	Let $\Lambda'$ be the attractor of $\Phi'$. Then
	$$\dim_H\pi'\Lambda'\geq\dim_H\pi'\widetilde{\Lambda}^1\geq\dim_H\pi'\widetilde{\Lambda}-\delta\geq\dim_H\pi_2\widetilde{\Lambda}-2\delta=\dim_H\widetilde{\Lambda}-2\delta\geq\dim_H\Lambda^1-3\delta>1.$$
	
	 Let $\mu'$ be the HHSS measure on $\Lambda'$ with weights $\frac{1}{\sharp\left(\widetilde{\Phi}_1\right)^m+1}$ for the functions in $\left(\widetilde{\Phi}_1\right)^m$ and weights $\frac{1}{2(\sharp\left(\widetilde{\Phi}_1\right)^m+1)}$ for the functions $f_{\il\hbar_1}^k,f_{\il\hbar_2}^k$. Thus, $\pi'\mu'$ is the natural self-similar measure on $\pi'\Lambda'$ and therefore, $1<\dim_H\pi'\Lambda'=\dim_H\pi'\mu'$. Because of the exact overlap and the fact that $\spt\pi'\mu'=\pi'\Lambda'$ cannot be contained in a line, $\spt\mu'$ cannot be contained in a plane. The exact overlap and $\dim_H\mu'\leq\dim_H\Lambda^1<2$ imply $\dim_H\mu'>\dim_H\pi'\mu'$,  which had to be proven.
\end{proof}

By changing the coordinates, without loss of generality we may assume that the projection in Proposition~\ref{pgoodmeasure} is a coordinate projection $\pi:(x,y,z)\mapsto(x,y)$. Moreover, since the measure $\mu$ in Proposition~\ref{pgoodmeasure} cannot be contained in any plane, we may assume that $\spt\mu$ is supported on an octant, separated away from the $z$ axis by restricting $\mu$ to a cylinder set.


Let us denote the projection along geodesics on $S^2$ to $S^1$ by $\gamma$. 
We note that $\gamma$ is well defined except on the poles. On the other hand,
$\gamma\circ P_3=P_2\circ\pi$.

Let $\nu:=(P_3)_*\mu$. Thus, $\gamma_*\nu=(P_2)_*\pi\mu$. By convenience, we use the  cylindrical coordinates in $\R^3$ and the radial coordinates on $\R^2$. That is, for $\R^3\ni\x=(r,\varphi,z)$, $\pi(\x)=(r,\varphi)$, $P_2(\pi(\x))=\varphi=\gamma(P_3(\x))$. Let us denote the conditional measures of $\mu$ on $\pi^{-1}(r,\varphi)$ by $\mu_{\pi^{-1}(r,\varphi)}$, the conditional measures of $\pi\mu$ on $P_2^{-1}(\varphi)$ by $\pi\mu_{P_2^{-1}(\varphi)}$, and the conditional measures of $\nu$ on $\gamma^{-1}(\varphi)$ by $\nu_{\gamma^{-1}(\varphi)}$, see Figure~\ref{fig:proj}.

\begin{figure}
	\centering
	\includegraphics[width=0.4\linewidth]{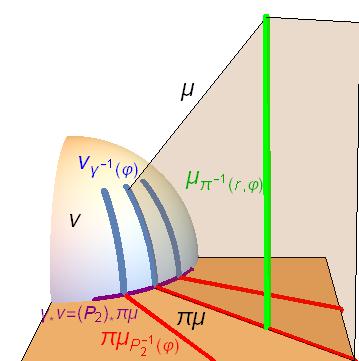}
	\caption{The conditional and projected measures along $P_2, P_3$ and $\gamma$.}\label{fig:proj}
\end{figure}

\begin{lemma}\label{lcondmeasure}
	For $\gamma_*\nu$-almost every $\varphi\in S^1$, $\dim_H\nu_{\gamma^{-1}(\varphi)}\geq\dim_H\mu-\dim_H\pi\mu>0$.
\end{lemma}

\begin{proof}
	By definition of conditional measures $\nu=\int\nu_{\gamma^{-1}(\varphi)}d\gamma_*\nu(\varphi)$. On the other hand, $\pi\mu=\int\pi\mu_{P_2^{-1}(\varphi)}d\gamma_*\nu(\varphi)$ and thus, $\mu=\int\mu_{\pi^{-1}(r,\varphi)}d\pi\mu(r,\varphi)=\iint\mu_{\pi^{-1}(r,\varphi)}d\pi\mu_{P_2^{-1}(\varphi)}(r)d\gamma_*\nu(\varphi)$. Hence,
	$$
	\nu=(P_3)_*\mu=\iint(P_3)_*\mu_{\pi^{-1}(r,\varphi)}d\pi\mu_{P_2^{-1}(\varphi)}(r)d\gamma_*\nu(\varphi).
	$$
	Since the conditional measures are uniquely defined up to a zero measure set
	\begin{equation}\label{econdmeas}
	\nu_{\gamma^{-1}(\varphi)}=\int(P_3)_*\mu_{\pi^{-1}(r,\varphi)}d\pi\mu_{P_2^{-1}(\varphi)}(r)\text{ for $\gamma_*\nu$-almost every $\varphi$.}
	\end{equation}
	Let us observe that for any compact line segment $I\subset\R^3$ which is not contained in any $1$ dimensional subspace of $\R^3$ the map $P_3:I\mapsto S^2$ is bi-Lipschitz. Hence, by Theorem~\ref{tdincon&lower}\eqref{edimc} and Proposition~\ref{pgoodmeasure} $$\dim_H(P_3)_*\mu_{\pi^{-1}(r,\varphi)}=\dim_H\mu_{\pi^{-1}(r,\varphi)}=\dim_H\mu-\dim_H\pi\mu>0\text{ for $\pi\mu$-a.e. $(r,\varphi)$.}$$
	By using the definition of Hausdorff dimension, let $A_{\varphi,n}$ be the set such that $\nu_{\gamma^{-1}(\varphi)}(A_{\varphi,n})>0$ and\linebreak $\dim_H\nu_{\gamma^{-1}(\varphi)}\geq\dim_HA_{\varphi,n}-\frac{1}{n}$. Thus, by \eqref{econdmeas} for $\gamma_*\nu$-a.e. $\varphi$ there exists a set $B_{\varphi,n}$ that
	$\pi\mu_{P_2^{-1}(\varphi)}(B_{\varphi,n})>0$ and for $\pi\mu_{P_2^{-1}(\varphi)}$-a.e. $r\in B_{\varphi,n}$
	$$(P_3)_*\mu_{\pi^{-1}(r,\varphi)}(A_{\varphi,n})>0.$$
	Hence,
	$$\dim_H\nu_{\gamma^{-1}(\varphi)}+\frac{1}{n}\geq\dim_HA_{\varphi,n}\geq\dim_H(P_3)_*\mu_{\pi^{-1}(r,\varphi)}=\dim_H\mu-\dim_H\pi\mu>0\text{ for $\gamma_*\nu$-a.e. $\varphi$.}$$
	Since $n$ was arbitrary, the proof is complete.
\end{proof}

\begin{proof}[Proof of Theorem~\ref{tradproj3}]
	Let $\mu$ and $\pi$ be as in Proposition~\ref{pgoodmeasure}. Since $\pi\mu$ is a HHSS measure satisfying SSC, we can apply Corollary~\ref{cradproj} and therefore, $$\dim_H\gamma_*\nu=\dim_H(P_2)_*\pi\mu=\min\left\{1,\dim_H\pi\mu\right\}=1.$$
	By Lemma~\ref{lcondmeasure}
	$$ 
	\dim_H\nu_{\gamma^{-1}(\varphi)}\geq\dim_H\mu-\dim_H\pi\mu>0.$$
	Thus, by \cite[Lemma~6.13]{H1}
	$$\dim_H(P_3)_*\mu=\dim_H\nu\geq\dim_H\gamma_*\nu+\dim_H\nu_{\gamma^{-1}(\varphi)}>1.$$
\end{proof}

\section{Proof of the main theorems}

\indent

In this section we show the remaining proofs.

\begin{proof}[Proof of Theorem~\ref{tmain}]
	Let $\Lambda$ be an HSS set in $\R^3$ such that it is not contained in any plane and $\dim_H\Lambda>1$. Moreover, let $g:\R^3\mapsto\R$ be a $C^1$ function satisfying the assumptions~\eqref{tcond1}-\eqref{tcond3}.  Since $\Lambda$ is compact, there exists an open neighbourhood of $\Lambda$ that $\|\nabla_{\x}g\|>0$ on the neighbourhood. By considering a sufficiently small cylinder of $\Lambda$ we may assume that there exists a ball $B$ that $\Lambda\subseteq B$ and $\|\nabla_{\x}g\|>0$ for every $\x\in B$. Let $f:\x\in B\mapsto\frac{\nabla_{\underline{x}}g}{\|\nabla_{\underline{x}}g\|}$. By assumption~\eqref{tcond2}, for every $t\in\R$ such that $t\cdot\x\in B$, $f(\x)=f(t\cdot\x)$. Thus, by assumption~\eqref{tcond3}, for any $\mu$ HSS measure with $\spt\mu\subseteq\Lambda$
	$$\dim_Hf_*\mu=\dim_H(P_3)_*\mu.$$
	
	It is enough to show the lower bound. Let $\mu$ be the HHSS measure as in Theorem~\ref{tradproj3}. Then by Theorem~\ref{tHochman}
	$$\dim_Hg(\Lambda)\geq\dim_Hg_*\mu\geq\mu-\essinf_{\x}\dim_H\pi_{g,\x}\mu,$$
	where we recall that $\pi_{g,\x}(\underline{y})=\frac{<\nabla_{\x}g,\underline{y}>}{\|\nabla_{\x}g\|}$.
	By Proposition~\ref{texcept}
	$$\dim_H\left\{\n\in S^2:\dim_H\pi_{\n}\mu<1\right\}\leq1.$$
	But by Theorem~\ref{tradproj3} $\dim_Hf_*\mu=\dim_H(P_3)_*\mu>1$, thus,
	$$f_*\mu(\left\{\n\in S^2:\dim_H\pi_{\n}\mu<1\right\})=0.$$
	And therefore $\mu-\essinf_{\x}\dim_H\pi_{g,\x}\mu=1$.
\end{proof}

\begin{proof}[Proof of Theorem~\ref{tdistance}]
	If $\Lambda$ is contained in a plane then we refer to Corollary~\ref{cdist} or \cite[Theorem~1.2]{O}. So we may assume that $\Lambda$ is not contained in any plane.
	
	By shifting $\Lambda$ we may assume that $\underline{0}\in\Lambda$. Let $\Lambda_{\il}$ be a cylinder set such that $\mathrm{dist}(\underline{0},\Lambda_{\il})>0$. Then $g(\x):=\|\x\|$ satisfies the conditions of Theorem~\ref{tmain} with self-similar set $\Lambda_{\il}$. Thus,
		$$\dim_HD_{\underline{0}}(\Lambda)\geq\dim_Hg(\Lambda)\geq\dim_Hg(\Lambda_{\il})=\min\left\{1,\dim_H\Lambda_{\il}\right\}=\min\left\{1,\dim_H\Lambda\right\}=1.
		$$
\end{proof}

\begin{proof}[Proof of Theorem~\ref{tprod}]
	Let $\Lambda$ be an arbitrary self-similar set on $\R$ that $\dim_H\Lambda>1/3$ with IFS~$\left\{\lambda_i x+t_i\right\}_{i=1}^q$, where $\lambda_i\in(-1,1)$. By applying \cite[Proposition~6]{PS} there exists a self-similar set $\Lambda'\subseteq\Lambda$ in $\R$ that $\dim_H\Lambda'>1/3$ with IFS~$\left\{\lambda x+t_i'\right\}_{i=1}^{q'}$, where $\lambda\in(0,1)$. Since $\Lambda'$ is not a singleton, there exists a cylinder set $\Lambda'_{\il}$ such that every element of $\Lambda'_{\il}$ is either strictly positive or strictly negative.
	
	It is easy to see that $\Lambda'_{\il}\times\Lambda'_{\il}\times\Lambda'_{\il}$ is an HSS set in $\R^3$ separated away from planes determined by the axes. Thus it is contained in one of the octants. Moreover, $\dim_H\Lambda'_{\il}\times\Lambda'_{\il}\times\Lambda'_{\il}>1$.
	
	Let $g(x,y,z)=xyz$. Then
	$$\nabla_{\x}g=\left(\begin{array}{c}
	yz \\
	xz   \\
	xy
	\end{array}\right).
	$$
	It is easy to see that $\nabla_{\x}g$ satisfies the assumptions~\eqref{tcond1} and \eqref{tcond2} of Theorem~\ref{tmain} on $\Lambda'_{\il}\times\Lambda'_{\il}\times\Lambda'_{\il}$.
	
	To show that $g$ satisfies \eqref{tcond3} of Theorem~\ref{tmain}, observe that there exists an open, simply connected set $V$ in $S^2$ such that $P_3(\Lambda'_{\il})\subset V$ and $V$ is uniformly separated away from the planes $x=0,y=0,z=0$. Since $\x\mapsto\nabla_{\x}g$ is one-to-one on every open octant and $\det(H_{\x}g)=2xyz\neq0$ for any $(x,y,z)\in V$, where $H_{\x}g$ denotes the Hesse matrix of $g$, we get that $\x\mapsto\nabla_{\x}g$ is a diffeomorphism between $V$ and its image $\nabla_{V}g$. Now, let $(\varphi,\theta)\mapsto\x(\varphi,\theta)$ be the natural parametrization of $V$. Thus,
	\begin{multline*}
	\langle\frac{\partial}{\partial\varphi}\nabla_{\x}g\times\frac{\partial}{\partial\theta}\nabla_{\x}g,\nabla_{\x}g\rangle=\langle\det(H_{\x}g)\left((H_{\x}g)^T\right)^{-1}\frac{\partial\x}{\partial\varphi}\times\frac{\partial\x}{\partial\theta},\nabla_{\x}g\rangle=\\
	\langle\frac{\partial\x}{\partial\varphi}\times\frac{\partial\x}{\partial\theta},\det(H_{\x}g)\left((H_{\x}g)\right)^{-1}\nabla_{\x}g\rangle=\frac{1}{2}\det(H_{\x}g)\langle\frac{\partial\x}{\partial\varphi}\times\frac{\partial\x}{\partial\theta},\x\rangle\neq0
	\end{multline*}
	for every point $\x\in V$ Hence, the normal vector of $S^2$ at $\nabla_{\x}g/\|\nabla_{\x}g\|$ is uniformly transversal to the normal vector of $\nabla_{V}g$ at the point $\nabla_{\x}g$. Thus, $P_3$ is a diffeomorphism between $\nabla_{V}g$ and $P_3(\nabla_{V}g)$, and therefore $h_g$ is bi-Lipsitz. 
			
	Thus, by Theorem~\ref{tmain}
	$$\dim_H\Lambda\cdot\Lambda\cdot\Lambda\geq\dim_H\Lambda'_{\il}\cdot\Lambda'_{\il}\cdot\Lambda'_{\il}=\dim_Hg(\Lambda'_{\il}\times\Lambda'_{\il}\times\Lambda'_{\il})=1.$$
\end{proof}

\begin{remark}
	Unfortunately, our method does not allows us to prove similar statements if $\dim_H\Lambda\leq1$. The method depends on dimension of the exceptional directions of orthogonal projections from $\R^3$ to $\R$. By using Hochman's result Theorem~\ref{tHochman}
	$$\dim_Hg_*\mu\geq\mu-\essinf_{\underline{x}}\dim_H\pi_{g,\underline{x}}\mu.$$
	On the other hand, in the case self-similar sets
	\begin{equation*} 
	\dim_H\left\{\pi\in\Pi_{3,1}:\dim_H\pi\Lambda<\min\left\{1,\dim_H\Lambda\right\}\right\}\leq1,
	\end{equation*}
	see Theorem~\ref{texcept}. Hence, to prove that the dimension does not drop, it is enough to show that $\dim_Hf_*\mu>1$, where $f:\underline{x}\mapsto P_3(\nabla_{\underline{x}}g)$. However, it is not possible if $\dim_H\Lambda\leq1$ and in particular if $\dim_H\mu\leq1$.
\end{remark}
	
\begin{remark}
	Conditions~(2) and (3) in Theorem~\ref{tmain} imply that we have to check only that $\dim_H(P_3)_*\mu>1$. These conditions seems rather technical, and we conjecture that they can be replaced by some more natural condition.
\end{remark}

\begin{acknowledgement}
	The author would like to express his gratitude to the referee for the helpful comments.
\end{acknowledgement}

\end{document}